\documentclass{elsarticle}
\usepackage{
amscd,
amsmath,
amssymb,
array,
chngpage,
color,
dsfont,
enumerate,
float,
graphicx,
latexsym,
mathrsfs,
multicol,
multirow,
verbatim,
wrapfig
}
\usepackage[UKenglish]{babel}
\usepackage[UKenglish]{isodate}
\cleanlookdateon
\usepackage[all]{xy} 

\newcommand{\use}{\mathit{use}}

\newcommand{\abs}[1]{\lvert#1\rvert}

\newcommand{\la}{\langle}
\newcommand{\ra}{\rangle}
\newcommand{\restrict}{\upharpoonright}

\newtheorem{thm}{Theorem}{\bf }{\it }
\newtheorem{prop}[thm]{Proposition}{\bf }{\it }
\newtheorem{pro}[thm]{Proposition}{\bf }{\it }
{\bf }{\it }
\newtheorem{cor}[thm]{Corollary}{\bf }{\it }
\newtheorem{df}[thm]{Definition}{\bf }{\rm }
\newtheorem{rem}[thm]{Remark}{\bf }{\rm }
{\bf }{\it }
{\bf }{\it }
\newtheorem{que}[thm]{Question}{\bf }{\it }
\newenvironment{proof}{\noindent {\bf Proof.}}{\nolinebreak \ \nolinebreak
\hfill{} \nolinebreak $\square$}

\newcommand{\darrow}{\!\downarrow}

\newcommand{\restr}{\mbox{\raisebox{.5mm}{$\upharpoonright$}}}

\newcommand{\bigset}[1]{\big\{ #1 \big\}}

\renewcommand{\leq}{\leqslant}
\renewcommand{\geq}{\geqslant}

\newcommand{\A}{\mathcal{A}}

\newcommand{\COMP}{{\rm REC}}
\newcommand{\REC}{{\rm REC}}
\newcommand{\CE}{{\rm RE}}

\newcommand{\Real}{{\mathds R}}

\DeclareMathOperator{\DNR}{DNR}
\DeclareMathOperator{\Parity}{Parity}

\begin{document}

\title{Covering the recursive sets}
\author[HAWAII]{Bj\o rn Kjos-Hanssen\fnref{GRANTS}}
  \ead{bjoern.kjos-hanssen@hawaii.edu}
\author[SINGAPORE]{Frank Stephan\fnref{GRANTS}} \ead{fstephan@comp.nus.edu.sg}
\author[NIJMEGEN]{Sebastiaan A.\ Terwijn\fnref{GRANTS}} \ead{terwijn@math.ru.nl}
\address[HAWAII]{Department of Mathematics,
University of Hawaii at Manoa, Honolulu, Hawaii 96822, USA}
\address[SINGAPORE]{Department of Mathematics and
Department of Computer Science, National University of Singapore,
10 Lower Kent Ridge Road, Block S17, Singapore 119076, Republic of Singapore}
\address[NIJMEGEN]{Department of Mathematics, Radboud University Nijmegen,
P.O.\ Box 9010, 6500 GL Nijmegen, The Netherlands}
\fntext[GRANTS]{This
work was partially supported by a grant from the Simons Foundation
(Grant number 315188 to
Bj\o rn Kjos-Hanssen), the National Science Foundation (Grant number
NSF-DMS-1545707 to Bj\o rn Kjos-Hanssen) 
and by a grant from the Ministry of Education in Singapore
(number MOE2013-T2-1-062 / R146-000-184-112 to F.~Stephan).
A substantial part of the work was performed while the first and third authors
were supported by
the Institute for Mathematical Sciences of the National University of Singapore
during the workshop on \emph{Algorithmic Randomness} from 2 to 30 June 2014.
}


\begin{abstract}
\noindent
We give solutions to two of the questions
in a paper by Brendle, Brooke-Taylor, Ng and Nies.
Our examples derive from a 2014 construction by Khan and Miller as well as
new direct constructions using martingales.

At the same time, we introduce the concept of i.o.\ subuniformity and relate
this concept to recursive measure theory.
We prove that there are classes closed downwards under
Turing reducibility that have recursive measure zero and
that are not i.o.\ subuniform. This shows that there are examples
of classes that cannot be covered with methods other than
probabilistic ones. It is easily seen that every
set of hyperimmune degree can cover the recursive sets.
We prove that there are both examples of hyperimmune-free degree
that can and that cannot compute such a cover.
\end{abstract}


\maketitle

\section{Introduction}

\noindent
An important theme in set theory has been the study of cardinal
characteristics. 
As it turns out, in the study of these there are certain analogies with 
recursion theory, where the recursive sets correspond to sets in the ground
model. Recently,
Brendle, Brooke-Taylor, Ng and Nies~\cite{Brendle.Brooke.ea:14}
point out analogies between cardinal characteristics and 
the study of algorithmic randomness.  
We address two questions raised in this paper that are connected to 
computing covers for the recursive sets. 

In the following, we will assume that the reader is familiar with 
various notions from computable measure theory, in particular, 
with the notions of Martin-L\"of null, Schnorr null and Kurtz null set. 
For background on these notions we refer the reader to
the books of Calude \cite{Calude},
Downey and Hirschfeldt~\cite{DH},
Li and Vit\'anyi \cite{LV} and Nies \cite{Ni09}.

Our notation from recursion theory is mostly standard, except for the
following: The natural numbers are denoted by $\omega$,
$2^\omega$ denotes the Cantor space and
$2^{<\omega}$ the set of all finite binary sequences.
${\mathbb R}^{\geq 0}$ denotes the set of
those real numbers which are not negative.
We denote the concatenation of strings $\sigma$ and $\tau$ by $\sigma\tau$.
The notation $\sigma\sqsubseteq\tau$ denotes that the finite
string $\sigma$ is an initial segment of the (finite or infinite)
string $\tau$. We identify sets $A\subseteq\omega$ with
their characteristic sequences, and $A\restr n$ denotes the
initial segment $A(0)\ldots A(n-1)$.
We use $\lambda$ to denote the empty string.
Throughout, $\mu$ denotes the Lebesgue measure on $2^\omega$.
We write $a\simeq b$ if either both sides are undefined, or they are both defined and equal.
We let $\Parity(x)=0$ if $x$ is even, and $\Parity(x)=1$ if $x$ is odd.

\begin{df}
A function $M:2^{<\omega}\rightarrow \Real^{\geq 0}$ is a \emph{martingale}
if for every $x \in 2^{<\omega}$, $M$ satisfies the averaging condition
\[
2M(\sigma) = M(\sigma0) + M(\sigma1),
\]
A martingale $M$ \emph{succeeds on} a set $A$ if
\[
\limsup_{n\rightarrow\infty} M(A\restr n)=\infty.
\]
The class of all sets on which $M$ succeeds is denoted by $S[M]$.
\end{df}

\noindent
For more background material on recursive martingales we refer the reader
to the above mentioned textbooks \cite{Calude,DH,LV,Ni09}.
The following definition is taken from
Rupprecht~\cite{Rupprecht,Rupprechtthesis}.

\begin{df}
An oracle $A$ is {\em Schnorr covering\/} if the union of all Schnorr null
sets is Schnorr null relative to $A$. An oracle $A$ is
{\em weakly Schnorr covering\/} 
if the set of recursive reals is Schnorr null relative to $A$.
For the latter, we will also say that $A$ Schnorr covers $\REC$. 
\end{df}

\begin{df}
A \emph{Kurtz test} relative to $A$ is an $A$-recursive sequence
of closed-open sets $G_i$ such that each $G_i$ has measure at most $2^{-i}$;
these closed-open sets are given by explicit finite lists of strings and
they consist of all members of $\{0,1\}^\omega$ extending one of the
strings. Note that $i \rightarrow \mu(G_i)$ can be computed relative
to $A$. The intersection of a Kurtz test (relative to $A$) is called a
Kurtz null set (relative to $A$). An oracle $A$ is \emph{Kurtz covering}
if there is an $A$-recursive array $G_{i,j}$ of closed-open sets such
that each $i$-th component is a Kurtz test relative to $A$
and every unrelativised Kurtz test describes a null-set contained in
$\cap_j G_{i,j}$ for some~$i$; $A$ is \emph{weakly Kurtz covering} if
there is such an array and each recursive sequence is contained in
some $A$-recursive Kurtz null set $\cap_j G_{i,j}$.
\end{df}

\noindent
Brendle, Brook-Taylor, Ng and Nies \cite{Brendle.Brooke.ea:14} called
the notion of (weakly) Schnorr covering in their paper (weakly) Schnorr
{\em engulfing\/}. In this paper, we will use the original terminology
of Rupprecht~\cite{Rupprecht,Rupprechtthesis}.
We have analogous notions for the other notions of effective null sets. 
As mentioned above, a set $A$ is weakly Kurtz covering if
the set of recursive reals is Kurtz null relative to~$A$.
We also have Baire category analogues of these notions of covering:
A set $A$ is {\em weakly meager covering\/} if it computes a meager set that
contains all recursive reals; more precisely, $A$ is weakly meager covering
iff there is an $A$-recursive
function $f$ mapping each binary string $\sigma$ to an extension $f(\sigma)$
such that every recursive sequence $B$ has only finitely many prefixes $\sigma$
for which $f(\sigma)$ is also a prefix of $B$.
Recall that a set $A$ is diagonally nonrecursive ($\DNR$) if there is a
function $f \leq_T A$ such that, for all $x$, if $\varphi_x(x)$ is defined
then $\varphi_x(x) \neq f(x)$. 
A set $A$ has hyperimmune-free Turing degree 
if for every $f \leq_T A$ there is a recursive function $g$
with $\forall x \, [f(x) \leq g(x)]$.

\section{Solutions to Open Problems}

\noindent
Brendle, Brooke-Taylor, Ng and Nies
\cite[Question 4.1]{Brendle.Brooke.ea:14}, 
posed three questions,  
(7), (8) and (9). In this section, we will provide the
answers to the questions (7) and (9). For this we note that
by \cite[Theorem 4.2]{Brendle.Brooke.ea:14} /
\cite[Corollary VI.12]{Rupprechtthesis} and \cite[Theorem 5.1]{KMS},
we have the following result.

\begin{thm} \label{characterize-wme}
A set $A$ is weakly meager covering iff it is high or of $\DNR$ degree.
\end{thm}

\noindent
We recall the following well-known definitions and results.

\begin{df}
A function $\psi$, written $e\mapsto (n\mapsto \psi_e(n))$,
is a recursive numbering if the function
$(e,n)\mapsto \psi_e(n)$ is partial recursive.
For a given recursive numbering $\psi$ and a function $h$,
we say that $f$ is $\DNR_h^\psi$ if for all $n$,
$f(n)\ne \psi_n(n)$ and $f(n)\le h(n)$.
An \emph{order function} is a recursive, nondecreasing, unbounded function.
\end{df}

\begin{thm}[{Khan and Miller \cite[Theorem 4.3]{Khan.Miller:14}}]\label{KhanMiller}
For each recursive numbering $\psi$ and for each order function $h$,
there is an $f \in \DNR^\psi_h$ such that $f$ computes no Kurtz random real.
\end{thm}

\noindent
Wang (see \cite[Theorem 7.2.13]{DH}) gave a martingale characterisation of 
Kurtz randomness.
While it is obvious that weakly Kurtz covering implies weakly Schnorr covering
for the martingale notions, some proof is needed in the case that
one uses tests (as done here).

\begin{pro}\label{engulf}
If $A$ is weakly Kurtz covering then $A$ is weakly Schnorr covering.
\end{pro}

\begin{proof}
Suppose $A$ is weakly Kurtz covering, as witnessed by the $A$-recursive
array of closed-open sets $G_{i,j}$. Then the sets $F_j = \cup_i G_{i,i+j+1}$
form an $A$-re\-cursive Schnorr test, as each $F_j$ has at most the measure
$\sum_j 2^{-i-j-2} = 2^{-i-1}$ and the measures of the $F_j$ is uniformly
$A$-recursive as one can relative to $A$ compute the measure of each
$G_{i,i+j+1}$ and their sum is fast converging. As for each recursive set
there is an $i$ such that all $G_{i,i+j+1}$ contain the set, each recursive
set is covered by the Schnorr test.
\end{proof}

\begin{thm}\label{BjoernFrank}
There is a recursive numbering $\psi$ and an order function $h$ 
such that for each set $A$, if $A$ computes a function $f$ that is 
$\DNR^\psi_h$ then $A$ is weakly Kurtz covering.
\end{thm}

\begin{proof}
Fix a correspondence between strings and natural numbers $\text{num}:2^{<\omega}\rightarrow\omega$ such that
\[
2^{\abs{\sigma}} - 1 \le \text{num}(\sigma) \le 2^{\abs{\sigma}+1}-2.
\]
For instance, $\text{num}(\sigma)$ could be the position of $\sigma$ 
in the length-lexicographically lexicographic ordering of all strings as
proposed by Li and Vit\'anyi \cite{LV}.
Let $\text{str}(n)=\text{num}^{-1}(n)$ be the string representation
of the number $n$. Thus
\[
2^{\abs{\text{str}(n)}} - 1 \le \text{num}(\text{str}(n)) = n \le 2^{\abs{\text{str}(n)}+1}-2.
\]
Let $\varphi$ be any fixed recursive numbering, let
\[
\langle a,b\rangle = \text{num}(1^{\abs{\text{str}(a)}}0\text{str}(a)\text{str}(b))
\]
in concatenative notation. Let
$\psi_{2\langle e,n\rangle}(x) = \varphi_e(n)$ for any $x$
and $\psi_{2y+1}=\varphi_y$. Note that $\psi$ is an acceptable numbering.
Let $s(e,n)=2\langle e,n\rangle$.
Then if $f$ is $\DNR$ with respect to $\psi$ then $f$ has the following property with respect to $\varphi$:
\[
f(s(e,n))\ne \varphi_e(n).
\]
Indeed,
\[
f(s(e,n)) = f(2\langle e,n\rangle) \ne \psi_{2\langle e,n\rangle}(2\langle e,n\rangle) = \varphi_e(n).
\]
Moreover,
\begin{eqnarray*}
s(a,b) =2\langle a,b\rangle
\le 2 (2^{1+\abs{1^{\abs{\text{str}(a)}}0\text{str}(a)\text{str}(b)}})
&=& 8 (2^{\abs{1^{\abs{\text{str}(a)}}}}2^{\abs{\text{str}(a)}}2^{\abs{\text{str}(b)}} )
\\
= 8 (2^{\abs{\text{str}(a)}} 2^{\abs{\text{str}(a)}} 2^{\abs{\text{str}(b)}})
&\le& 8(a+1)^2(b+1).
\end{eqnarray*}
Consider a partition of $\omega$ into intervals $I_m$ such that 
$\abs{I_m}$ is $2+\log(m+1)$ rounded down, 
and let $h(m) = \abs{I_m}$.
If $f$ is $\DNR^\psi_h$ then we have
\[
\forall \varphi_e\, \forall n\,
(
f(s(e,n))\in\{0,1\}^{I_{s(e,n)}} \text{ and }f(s(e,n))\ne\varphi_e(n)
).
\]
Given a recursive set $R$, there is, by the fixed-point theorem,
an index $e$ such that, for all $n$,
$\varphi_e(n) = R \restrict I_{s(e,n)}$ and
$f(s(e,n)) \neq R \restrict I_{s(e,n)}$.
Note that for every fixed $e$,
\begin{eqnarray*}
    \prod_{n=0}^\infty (1-2^{-\abs{I_{s(e,n)}}}) \leq
    \prod_{n=e+2}^\infty (1-2^{-(2+\log(8(e+1)^2(n+1)+1))}) 
\\
    \le\prod_{n=e+2}^\infty (1-2^{-(3+\log(8(e+1)^2(n+1)))}) =
    \prod_{n=e+2}^\infty (1-2^{-(6+2\log(e+1)+\log(n+1))}).
\end{eqnarray*}
The last product in this formula is $0$, as the sum
$$
   \sum_{n=e+2}^\infty 2^{-(6+2\log(e+1)+\log(n+1))} =
   \frac1{64} \cdot (e+1)^{-2} \cdot \sum_{n=e+2}^\infty \frac1{n+1}
$$
diverges. Thus
\[
\mu( \{ B : \exists e\forall n \,[B\restrict I_{s(e,n)} \ne f(s(e,n))] \}) 
\leq 
\sum_e\prod_{n=0}^\infty (1-2^{-\abs{I_{s(e,n)}}}) = 0.
\]
So if $f$ is $A$-recursive then we have a $\Sigma^0_2(A)$ null set
that contains all recursive sets, as desired.
\end{proof}

\begin{thm}[{answer to 
	\cite[Question 4.1(7)]{Brendle.Brooke.ea:14}}]\label{47}
There exists a set $A$ satisfying the following conditions:
\begin{enumerate}
\item \label{WME} $A$ is weakly meager covering;
\item \label{BNSR} $A$ does not compute any Schnorr random set;
\item \label{HIF} $A$ is of hyperimmune-free degree;
\item \label{WSE} $A$ is weakly Schnorr covering.
\end{enumerate}
\end{thm}

\begin{proof}
Let $h$ and $\psi$ as in Theorem \ref{BjoernFrank}.
By Theorem \ref{KhanMiller}, there is an $f \in \DNR^\psi_h$ such that $f$ computes no
Kurtz random real. Let $A$ be a set Turing equivalent to $f$.
\begin{enumerate}
\item[\ref{WME}.] By Theorem \ref{characterize-wme}, $A$ is weakly meager covering. 
Alternatively, one could use the fact that every weakly Kurtz covering
oracle is also weakly meager covering and derive the item \ref{WME} from
the proof of item \ref{WSE}.
\item[\ref{BNSR}.] Since each Schnorr random real is Kurtz random,
$A$ does not compute any Schnorr random real.
\item[\ref{HIF}.] Since $A$ does not compute any Kurtz random real,
$A$ is of hyperimmune-free degree.
\item[\ref{WSE}.] By Theorem \ref{BjoernFrank}, $A$ is weakly Kurtz covering.
In particular, by Proposition \ref{engulf}, $A$ is weakly Schnorr covering.
\end{enumerate}
This completes the proof.
\end{proof}

\medskip
\noindent
The following proposition is well-known and will be used in 
various proofs below.

\begin{prop} \label{prop:tt}
If $A$ is of hyperimmune-free Turing degree and $B\leq_T A$ then $B\leq_{tt}A$.
\end{prop}

\noindent
Franklin and Stephan \cite{FS10} gave the following characterisation:
A set $A$ is \emph{Schnorr trivial} iff for every $f \leq_{tt} A$
there is a recursive function $g$ such that, for all $n$,
$f(n) \in \{g(n,0),g(n,1),\ldots,g(n,n)\}$; this characterisation
serves here as a definition.

\begin{thm}[{answer to 
	\cite[Question 4.1(9)]{Brendle.Brooke.ea:14}}] \label{firstStatement}
There is a hyperimmune-free oracle $A$ which is not $\DNR$ (and thus
low for weak $1$-genericity) and which is not Schnorr trivial
and which does not Schnorr cover all recursive sets.
\end{thm}

\begin{proof}
We show that there is a set $A$ such that the following conditions hold:
\begin{enumerate}
\item $A$ is not $\DNR$;
\item $A$ does not have hyperimmune degree;
\item $A$ is not Schnorr trivial;
\item $A$ is not weakly Schnorr covering.
\end{enumerate}
To this end, a partial-recursive $\{0,1\}$-valued function $\psi$
is constructed such that every total extension $A$ of
hyperimmune-free degree satisfies the conditions that
$A$ is not $\DNR$, not Schnorr trivial and not weakly Schnorr covering.
The property that $A$ is not Schnorr trivial is obtained
by showing that there is an $A$-recursive
function $f$ such that $C(f(x)) > x$ for infinitely many~$x$.
(Here, $C$ denotes the plain Kolmogorov complexity. In what follows, 
by Kolmogorov complexity we will always mean the plain complexity.)
The property that $A$ is not weakly Schnorr covering will be obtained by
showing that there is no martingale tt-reducible to $A$ and no
recursive bound such that the martingale Schnorr succeeds on all
recursive sets using this bound. 
Note that since $A$ is of hyperimmune-free degree,
by Proposition~\ref{prop:tt}
it is sufficient to consider tt-reductions instead of Turing-reductions here. 

The basic idea is to construct the partial recursive
function $\psi$ such that its domain at every stage $s$
is the complement of the currently active intervals $I_n$.
Here $I_0 = \{0,1\}$ and, for $n>0$, $I_n = \{2^n+1,2^n+2,\ldots,2^{n+1}\}$.
When $\psi$ becomes defined on some interval $I_n$ by setting it nonactive,
$\psi$ takes on $I_n$ a characteristic function
$\sigma \in \{0,1\}^{I_n}$ which has not
been killed previously by the construction.\footnote{If $\sigma$ is killed then the infinite
branches of the tree extending $\sigma$ are cut off so that the choice to be killed will
no longer be available in the future.}
At each stage $t$ the following activities will be carried out:
\begin{itemize}
\item Select the requirement of highest priority which
needs attention and is permitted to act;
\item For the reserved interval $I_n$, find the next
interval $I_m$ which should be active and
which has to be so large that one can satisfy
the growth requirements of the martingale
to not succeed by making the $I_o$ with
$n<o<m$ to be non-active and by later killing
certain $\sigma \in \{0,1\}^{I_n}$ (see below);
\item Make $\psi$ defined on all intervals $I_o$ with
$n<o<m$;
\item Update the tree $T_t$ so that it takes the new $\psi_t$
but only those $\sigma$ killed before stage $t$ into account:
The tree $T_t$ has those infinite branches $\tilde A$ which extend
$\psi_t$ and which, on any active $I_n$, do not take a value
$\sigma$ which has been killed prior to stage $t$;
\item Kill every $\sigma \in \{0,1\}^{I_n}$ which needs to be
killed according to the selected requirement and which
has not been killed before (this depends on $T_t$);
\item Make $I_m$ to be the reserved interval for the requirement;
\item For all active $I_o$ with $o<t$ and all $\sigma \in \{0,1\}^{I_o}$,
if it is found within $t$ steps
that the conditional Kolmogorov complexity of $\sigma$ given $o$
is strictly below $2^o-1$ then kill $\sigma$
(if not already done so before);
\item Initialise and cancel requirements as needed for
requirements $R_{k',c'}$ with $k'<t$ and $c'\leq k'$.
\end{itemize}
An oracle $\tilde A$ is valid at $t$ iff it is an infinite branch
of $T_t$ and it is valid if it is an infinite branch
of $T = \cap_t T_t$. 
The tree $T$ will have infinite
hyperimmune-free branches $A$ and it will be shown that 
any such branch $A$ is neither Schnorr trivial 
nor $\DNR$ nor weakly Schnorr covering.

Now some more details are given for the requirements.
For these there is a list $(M_k,f_k)$ of martingales
$M_k$ given by truth-table reductions to oracles
and of a recursive bound functions $f_k$; though one
cannot avoid that partial truth-table reductions and
bound functions are in the list, one can nevertheless
make the list in a way that one can check for each
$\sigma,\ell$ and $t$ whether $M_k(\sigma)$ is defined within
$t$ steps and whether $f_k(0),\ldots,f_k(\ell)$ are all
defined within $t$ steps. Note that only the total $(M_k,f_k)$
are relevant and that the others will get stuck somewhere in
the construction and will be ignored by all sufficient large
instantiations of the requirements with true parameters. Now
$(M_k,f_k)$ succeeds on a recursive set $B$ iff there are infinitely
many $n$ such that $M_k(B(0)B(1)\ldots B(f_k(n))) > n$; it is
therefore a goal of the construction to prevent this from happening
and to construct together with $M_k$ a recursive set $B$ (depending
on $k$ as well) such that $(M_k^A,f_k)$ does not succeed on $B$.
For the construction, let $\use_k(x)$ denote the
first time $t>x$ is found such that $f_k$ is defined on all $y \leq x$
and $M_k^{\tilde A}(\sigma)$ is defined for all $\sigma$ up to length $x$
by querying only values below $t$, independently on which oracle
${\tilde A}$ is used; $\use_k(x) = \infty$ if some of the above mentioned
computations do not terminate. In the following, $t$ will always be
the number of the current stage and the requirements will explicitly
check that the use of those members of the list which are considered
to be valid is below the current stage number $t$ on the relevant
inputs.

One can without loss of generality
assume that each $M_k$ is a savings martingale so that it
never goes down by more than $1$ and that the functions
$f_k$ are strictly monotonically increasing; furthermore,
the functions and martingales are either total or defined up
to a certain point and undefined from then onwards.
Franklin and Stephan \cite{FS10}
provide more details on such type of martingales.
If ${\tilde A}$ is hyperimmune-free
then this list is sufficient to deal with all relevant
martingales as one can replace martingales by saving
martingales and then the bound by a recursive upper bound.
For each $k$ there are exactly $k+1$ many requirements
$R_{k,0},R_{k,1},\ldots,R_{k,k}$ for $(M_k,f_k)$.
The requirement $R_{k,c}$ is said to have
true parameters iff there are exactly $c+1$
indices $e \in \{0,1,\ldots,k\}$ with $M_e,f_e$
being total and $k$ is one of these.
Note that when $M_e,f_e$ are total then $M_e$
has to be a savings martingale as described above and
$f_e$ has to be a strictly monotonically increasing
recursive function.
At a stage $t$, a requirement $R_{k,c}$ can (a) be initialised,
(b) be cancelled, (c) require attention or (d) act.
\begin{description}
\item{Initialisation:}
A requirement $R_{k,c}$ can be initialised at stage $t$
and request an interval $I_q$
(so that $q$ denotes from now on the index of that interval which $R_{k,c}$
took while being initialised)
iff there are exactly $c$ numbers $k_0,k_1,\ldots,k_{c-1}$
with $k_0 < k_1 < \ldots < k_{c-1} < k$ such that the
following conditions hold:
\begin{itemize}
\item $q > k$ and $2^{|I_q|} \cdot r_k$ is an integer (for the
sequence of $r_k$ defined below);
\item for each $e \leq k$, $\use_e(\max(I_q)) < t$
iff $e \in \{k_0,k_1,\ldots,k_{c-1},k\}$;
\item for all $c'< c$,
the requirement $R_{k_{c'},c'}$ is currently
active and has reserved some interval $I_o$ with $o>q$;
\item $I_q$ is active and all intervals $I_o$ on which $R_{k,c}$ has
acted in prior stages satisfy $o<q$.
\end{itemize}
Let $E_{b,t}$ contain all oracles ${\tilde A}$
such that ${\tilde A}$ is on $T_t$ and
for all $D$ on $T_t$ which coincide with ${\tilde A}$
below the given bound $\use_k(b)$,
${\tilde A} \leq_{lex} D$, that is,
${\tilde A}$ is the least representative of the class
of oracles which do not differ below $\use_k(b)$ from ${\tilde A}$. Now let
$$
N_{k,b,t}(\sigma) =
\sum_{\mbox{${\tilde A} \in E_{b,t}$}} M_k^{\tilde A}(\sigma)
$$
for all $\sigma$ up to length $b$.
Now define $B$ up to the maximal value $x$ with $\use_k(x) < \min(I_q)$
as taken such that $N_{k,x,t}$ does not grow and let
$$
u = \max\{M^{\tilde A}_k(B(0)B(1)\ldots B(x)): {\tilde A} \in E_{x,t}\}.
$$
The values $u,x$ are updated and $B$ defined on more places when
the requirement acts.
\item{Cancellation:}
A requirement $R_{k,c}$ gets cancelled if there are more
than $c$ numbers $e < k$ 
for the initial interval $I_q$ on which $R_{k,c}$ got
initialised for the current run.
\item{Attention:}
Now let $r_0,r_1,\ldots$ be a recursive sequence of negative powers
of $2$ which converge from above to $0$ and have the property that
the sum of the $r_k \cdot (k+1)$ for $k=0,1,2,\ldots$ is $1/2$.
For a requirement $R_{k,c}$ currently having reserved an interval
$I_n$, recall that $x$ is the place up to which the recursive
set $B$ of the requirement has been defined
and $u$ is the maximal value $M^{\tilde A}(B(0)B(1)\ldots B(x))$ takes
for some ${\tilde A} \in T_t$ (assuming that $\psi$ does not get defined
on intervals below $I_n$ which will not happen in the case that
$R_{k,c}$ acts). Now the requirement $R_{k,c}$ needs attention if
there is an interval $I_m$ such that
\begin{itemize}
\item $I_n$ and $I_m$ are both active and $n<m$;
\item $R_{k,c}$ has currently reserved $I_n$;
\item $\use_e(\max(I_q))>t$ for all
$e \in \{0,1,\ldots,k-1\}-\{k_0,k_1,\ldots,k_{c-1}\}$
for the initial interval $I_q$ on which the current run
of the requirement $R_{k,c}$ was initialised;
\item the requirements $R_{k_{e},e}$ with $e<c$ have currently reserved
some interval $I_o$ with $o>m$;
\item the maximal value $x'$ with $\use_k(x') < \min(I_m)$
satisfies $x' > \max(I_n)$ and
$x' > f_k((u+1) \cdot 4^{\max(I_n)} / r_k)$.
\end{itemize}
\item{Acting:}
If $R_{k,c}$ receives attention, it acts as follows, where
the parameters $B,x,x'$ are as under the item ``Attention''.
\begin{itemize}
\item All $I_o$ with $n < o < m$ will be set non-active (if not
done before) and $\psi_t$ will be defined on these intervals
and $T_t$ will be updated as outlined above;
\item Let $u = \max \{M^{\tilde A}_k(B(0)B(1)\ldots B(x))$: $\tilde A$
is an infinite branch of $T_t\}$;
\item $N_{k,x',t}$ will be computed and $B$ will be extended from the
domain up to $x$ to the domain up to $x'$ in the way that
$N_{k,x',t}$ does not grow;
\item For each $\sigma \in \{0,1\}^{I_n}$ let $u_\sigma = \max\{
M^{\tilde A}_k(B(0)B(1)\ldots B(x'))$: $\tilde A$ is on $T_t$
and extends $\sigma\}$ --- once this is defined, one kills those
$r_k \cdot 2^{|I_n|}$ strings $\sigma \in \{0,1\}^{I_n}$
for which $u_{\sigma}$ is maximal (see next item);
\item Let $u'=\max\{M^{\tilde A}_k(B(0)B(1)\ldots B(x'))$:
${\tilde A}$ is on $T_t$ and ${\tilde A}$ restricted $I_n$
has not been killed in the previous step$\}$, that is, $u'$
is bounded by the value number $r_k \cdot 2^{|I_n|}+1$ in
a list of all the $u_\sigma$ considered, in descending order;
\item The new value of $x$ is the current $x'$ and the new interval
selected for $R_{k,c}$ is $I_m$ and the new value of $u$ is $u'$.
\end{itemize}
\end{description}
For the verification, it first
should be noted that for each $I_n$, at most
$2^{|I_n|}-1$ many $\sigma \in \{0,1\}^{I_n}$ get killed
and therefore the amount of $\sigma$ available is never
exhausted. The reason is that the requirements kill at
most $2^{|I_n|} \cdot \sum_{k,c} r_k = 2^{|I_n|-1}$ many
$\sigma$ and the Kolmogorov complexity condition at most
$2^{|I_n|-1}-1$ many $\sigma$, so at least one $\sigma$
remains. Thus the $\psi$ can on each $I_n$ get defined
when $I_n$ is set to be non-active and the tree $T_t$
has in each step and also in the limit infinite branches.
So there is a hyperimmune-free set $A$ on the tree $T$.

Second there are infinitely many $I_n$ which remain
active forever. Assume that it is shown that the interval
$I_n$ is never set to inactive. One can see that every interval
gets only finitely often reserved by a requirement
and therefore it happens only finitely often that a
requirement acts with the interval $I_n$ or a smaller
one being the reserved interval; when this has happened
for the last time, there is a larger interval $I_m$
such that $I_m$ is the least current active interval
above $I_n$. From now on, it only happens that some
interval $I_m$ or beyond will be the reserved interval
of a requirement which is going to act and therefore
$I_m$ will never be set inactive. So one can prove
by induction that infinitely many intervals will
remain active forever.

Third the resulting set is not Schnorr trivial as
there are infinitely many intervals $I_n$ from
$2^n+1$ to $2^{n+1}$ which remain active forever and
on them, $A$ restricted to $I_n$ has at least the
Kolmogorov complexity $2^n-1$ conditional to $n$;
thus the set $A$ is not Schnorr trivial.

Fourth, let $k_0,k_1,\ldots$ be the (noneffective) sublist
of all pairs $(M_k,f_k)$ such that $M_k$ is a total
truth-table reduction giving a savings martingale
and $f_k$ is a total recursive function.
Now one shows by induction over $c$ that each requirement
$R_{k_c,c}$ gets only finitely often cancelled
and is eventually permanently initialised and
acts infinitely often. Assume that the stage
$t$ is so large that the following conditions hold:
\begin{itemize}
\item the pair $(M_{k'},f_{k'})$ with
$k' \in \{0,1,\ldots,k_c\}-\{k_0,k_1,\ldots,k_c\}$
have reached their first undefined places and
let $y$ be the maximum of these places;
\item all cancellations of $R_{k_c,c}$ due to
these $k'$ have already occurred;
\item the requirements $R_{k_{c'},c'}$ with $c'<c$
are all initialised and will not be cancelled after
stage $t$ and will act infinitely often;
\item there is an active interval $I_n$ such that
neither $I_n$ nor any larger interval has so far
been reserved by $R_{k_c,c}$ and all requirements
$R_{k_{c'},c'}$ with $c'<c$ have currently reserved
some interval beyond $I_n$.
\end{itemize}
Then the requirement $R_{k_c,c}$ will be initialised,
for example on $I_n$; it will not be cancelled again.
Now one needs to show that it acts infinitely often;
assume by way of contradiction that the requirement
would remain forever on an interval $I_n$ without acting.
There is an interval $I_m$ beyond $I_n$ such that
$I_m$ is active forever and all the conditions of
the request of attention are satisfied except the
first one -- this is due to selecting an $I_m$
with sufficiently large index / position. Now,
by induction hypothesis, the requirements
$R_{k_{c'},c'}$ will act often enough so that
they eventually reserve intervals beyond $I_m$
and from that time onwards $R_{k_c,c}$ will
require attention and therefore eventually act.

Fifth: If $M_k,f_k$ are total and $c$ is chosen such
that $(k,c)$ are true parameters
then the set $B$ constructed by requirement $R_{k,c}$
in its infinite run is not covered by the martingale
$M_k$ with bound $f_k$ in the Schnorr sense.
There is a case distinction between the case
where, in a run, the requirement $R_{k,c}$ acts
for the first and for a subsequent time.

Assume now that $R_{k,c}$ acts for the first time.
Let $x$ be the value up to which $B$ has been defined
in the initialisation and let $u$ be the corresponding
maximum value taken by some martingale up to $x$ on $B$
which is valid at the time of initialisation. Note that
$\psi_t$ will be defined in stage $t$ on all values
strictly between $\max(I_n)$ and $\min(I_m)$. Then
$N_{k,x',t}$ is the sum of at most $2^{\max(I_n)}$ martingales
and $B$ is chosen on the values from $x+1$ up to $x'$ such
that $N_{k,x',t}$ does not increase; hence the value
$N_{k,x',t}(B(0)B(1)\ldots B(x'))$ is bounded
by $2^{\max(I_n)} \cdot u$ and therefore each
outgoing martingale $M_k^{\tilde A}$ satisfies
$$
M_k^{\tilde A}(B(0)B(1)\ldots B(x')) \leq 2^{\max(I_n)} \cdot u
$$
while at the same time $x'$ satisfies
$$
f_k((u+1) \cdot 4^{|I_n|} / r_k ) < x'
$$
and thus, for the new bound $u'$,
$$
f_k((u'+1) \cdot 2^{I_n} / r_k ) < x'
$$
which can be used as an incoming bound for subsequent actions 
of the requirement $R_{k,c}$.

If now $R_{k,c}$ acts for a subsequent time in the run of
a requirement, then one verifies besides the above assurance
on the outgoing bound -- the proof goes through unchanged --
also that the martingale cannot succeed in the Schnorr sense
between $x$ and $x'$ where $x'$ is the new point up to which
$B$ gets defined during the acting.

Now let $I_o$ be the interval on which
it acted before it acts on $I_n$ and let $I_m$ be the interval
where it scheduled to act next (though the interval might actually
be larger), that is $I_n$ and $I_m$ are the parameters used during
the current acting of the requirement. Note that $o < n < m$.
Furthermore, when the requirement acts on $I_n$ then $I_n$ is the
first active interval after $I_o$ and therefore $\psi$
is defined before stage $t$ between $\max(I_o)$ and $\min(I_n)$
and it will become defined between $\max(I_n)$ and $\min(I_m)$
during the stage $t$ or already before stage $t$.
Furthermore the sum of the martingales $N_{k,x',t}$
(with the parameters defined as in the proof) only need
to take into account the oracles which coincide with $\psi_t$
as only those can be identical with the $A$ as $A$ is on $T$.
Let $x$ be the bound to which $B$
is defined before $I_n$ acts and $x'$ be the bound
after $I_n$ acts; furthermore, $u$ and $u'$ are
defined accordingly and $t$ is the time when $R_{k,c}$
acts on $I_n$. By induction hypothesis,
$f_k((u+1) \cdot 2^{\max(I_o)} / r_k) < x$.
Therefore one has only to show that
$$
M^{\tilde A}_k(B(0)B(1)\ldots B(x')) < (u+1) \cdot 2^{\max(I_o)} / r_k
$$
in order to satisfy the constraint on non-success for all ${\tilde A}$ on $T_t$
which do not get ${\tilde A}$ restricted to $I_n$ killed in stage $t$.
This can be seen as $N_{k,t,x'}$ does not go up on $B$ from $x$ to $x'$
due to the choice of $B$ and furthermore there are
at most $2^{\max(I_o)}\cdot 2^{|I_n|}$ many initial segments
$\tau = {\tilde A}(0){\tilde A}(1)\ldots {\tilde A}(\min(I_m)-1)$
which have to be taken into account to compute the value of the sum
$N_{t,x',k}(B(0)B(1)\ldots B(x'))$;
among these values, the largest $r_k \cdot 2^{|I_n|}$ many
terms in the sum will be removed from it due to the killing of the
corresponding $\sigma$; it follows that the maximum $u'$
of the remaining terms satisfies
$$
u' \leq u \cdot \frac{2^{\max(I_o)} \cdot 2^{|I_n|}}{r_k \cdot 2^{|I_n|}}
\leq u \cdot 2^{\max(I_o)} / r_k
$$
which satisfies the required bound. This calculation is based on the
fact that if there are up to $a$ values whose
sum bounded by $u \cdot a$ and one kills the largest $b$ of these
values then the remaining values are each are bounded by
$u \cdot a/b$, as otherwise the $b$ killed values would
each be strictly above $u \cdot a/b$ and
have a sum strictly above $b \cdot u \cdot a/b = u \cdot a$
what is impossible by assumption on $u \cdot a$ being the
sum of all of the values. Now the $2^{\max(I_o)} \cdot 2^{|I_n|}$ in the
numerator is an upper bound on the overall number of terms to
be considered and $2^{|I_n|} \cdot r_k$ is a lower bound on the number
of largest terms to be removed from the sum which follows
from the overall number of new strings killed
in this iteration. Thus none of the surviving oracles $\tilde A$ satisfies
$$
M_k^{\tilde A}(B(0)B(1)\ldots B(x')) > u \cdot 2^{\max(I_o)} / r_k
$$
and so the martingale is below the value $u \cdot 2^{\max(I_o)} / r_k+1$
on all prefixes of $B(0)B(1)\ldots B(x')$.
Thus the growth bound is maintained between $x$ and $x'$.
In particular it follows that the growth bound on $M_k^A$ is 
maintained at every acting of the requirement except for the first after 
the initialisation. Therefore $(M_k^A,f_k)$ does not Schnorr succeed 
on the recursive set~$B$.

Sixth the set $A$ is not $\DNR$. To see this, recall
that for being $\DNR$ and hyperimmune-free there needs
to be a recursive function $g$ such that $A$ up
to $g(n)$ has at least Kolmogorov complexity $n$
for every $n$ \cite{KMS}; without loss of generality
$g$ can be taken to be strictly monotonically increasing.
There is a $k$ such that $f_k = g$
and $M^A_k(\sigma)$ a martingale which always bets $0$.
There is a corresponding requirement $R_{k,c}$ which
acts infinitely often. When acting with $I_n$ being
a reserved interval, the requirement ensures
that there is another interval $I_m$ such that
$f_k(2^{\max(I_n)} / r_k) < \min(I_m)$
and $\psi$ is defined between $\max(I_n)$ and
$\min(I_m)$. It follows that the Kolmogorov
complexity of $A$ up to $f_k(2^{\max(I_n)})$
is at most $\max(I_n)+n+ O(1)$ for this $I_n$
and infinitely many other $I_n$, thus the
constraint is violated and $A$ is not $\DNR$.
This completes the proof.
\end{proof}

\begin{rem}
The reader may object that the original question in \cite{Brendle.Brooke.ea:14} asked for a set that was 
\emph{not low for Schnorr tests} rather than \emph{not Schnorr trivial}. 
However, we can recall the following facts:
\begin{itemize}
	\item Kjos-Hanssen, Nies and Stephan \cite{KNS} showed that
        if $A$ is low for Schnorr tests then $A$ is low for Schnorr randomness;
	\item
	Franklin \cite{Fr08} showed that if $A$ is
        low for Schnorr randomness then $A$ is Schnorr trivial.
\end{itemize}
\end{rem}

\section{Infinitely Often Subuniformity and Covering}\label{sec:iosub}

\noindent
Let $\la . \, , .\ra$ denote a standard
recursive bijection from $\omega \times\omega$ to $\omega$.
For a function $P: \omega \rightarrow \omega$ define
\[
P_n(m) = P(\la n,m \ra )
\]
and say that $P$ \emph{parametrises} the class of functions
$\{P_n: n\in\omega\}$. We identify sets of natural numbers
with their characteristic functions. A class $\A$ is
\emph{(recursively) uniform} if there is a recursive function $P$
such that $\A = \{P_n: n\in\omega\}$, and \emph{(recursively) subuniform}
if $\A \subseteq \{P_n: n\in\omega\}$.
These notions relativise to any oracle $A$ to yield the notions of
\emph{$A$-uniform} and \emph{$A$-subuniform}.

It is an elementary fact of recursion theory that the recursive 
sets are not uniformly recursive.
The following theorem, as cited in Soare's book \cite[page 255]{Soare},
quantifies exactly how difficult it is to do this:

\begin{thm}[Jockusch] \label{Jockusch}
The following conditions are equivalent:
\begin{enumerate}[\rm (i)]
\item $A$ is high, that is, $A' \geq_T \emptyset''$,
\item the recursive functions are $A$-uniform,
\item the recursive functions are $A$-subuniform,
\item the recursive sets are $A$-uniform.
\end{enumerate}
\noindent
If $A$ has r.e.\ degree then {\rm (i)--(iv)} are each equivalent to:
\begin{enumerate}
\item[\rm (v)] the recursive sets are $A$-subuniform.
\end{enumerate}
\end{thm}

\noindent
In the following we study infinitely often parametrisations and 
the relation to computing covers for the recursive sets.

\subsection{Infinitely Often Subuniformity}

\noindent
The next definition generalises from ``Schnorr covering'' to
``covering'' which just says that a martingale succeeds on all
sets of the class
(without having a bound on the time until it has to succeed
infinitely often).

\begin{df} \rm
We say that a set $X$ \emph{covers} a class $\A$ if there is an
$X$-recursive martingale $M$ such that $\A\subseteq S[M]$.
\end{df}

\noindent
Note that for $X$ recursive this is just the definition
of \emph{recursive measure zero}: $\A$ has recursive measure zero if there is a recursive martingale $M$ such that $\A\subseteq S[M]$.

\begin{df} \rm
A class $\A\subseteq 2^\omega$ is called \emph{infinitely often subuniform}
(i.o.\ subuniform for short) if there is a recursive function
$P\in\{0,1,2\}^\omega$ such that
\[
\forall A\in\A \; \exists n \; \big[ \exists^\infty x \big(P_n(x) \neq 2\big)
\wedge \forall x \big( P_n (x)\neq 2 \rightarrow P_n(x) = A(x) \big) \big].
\]
\end{df}
\noindent
That is, for every $A\in\A$ there is a row of $P$ that
computes infinitely many elements of $A$ without making mistakes.
Again, we can relativise this definition to an arbitrary set $X$:
A class $\A$ is i.o.\ $X$-subuniform if $P$ as above is
$X$-recursive.

Let $\COMP$ denote the class of recursive sets.
Recall that $A$ is a PA-complete set if $A$ can
compute a total extension of every $\{0,1\}$-valued
partial recursive function.
Note that if a set $A$ is PA-complete then
$\COMP$ is $A$-subuniform (see Proposition~\ref{implications} below).

For every recursive set $A$ there is a recursive 
set $\hat A$ such that $A$ can be reconstructed from any
infinite subset of $\hat A$. Namely, let $\hat A(x)=1$ precisely when
$x$ codes an initial segment of $A$.
So it might seem that
any i.o.\ sub-parametrisation of $\COMP$ can be converted
into a sub-parametrisation
in which every recursive set is completely represented.
However, we cannot do this uniformly (since we cannot get rid of the
rows that have $P_n(x)=2$ a.e.) and indeed the implication
does not hold.

\begin{prop} \label{implications}
We have the following picture of implications:
$$
\begin{array}{ccccc}
&& \mbox{$A$ is high} & \Rightarrow & \mbox{$A$ has hyperimmune} \\
&&&& \mbox{degree} \\
&& \raisebox{0pt}[16pt][10pt]{$\Downarrow$}  
&& \raisebox{0pt}[16pt][10pt]{$\Downarrow$}   \\
\mbox{$A$ is PA-complete} & \Rightarrow &
\COMP \mbox{ is $A$-subuniform} & \Rightarrow & \mbox{$\COMP$ is i.o.} \\
&&&& \mbox{$A$-subuniform}
\end{array}
$$
No other implications hold than the ones indicated.
\end{prop}

\begin{proof}
The proposition follows from the following observations.

If $A$ is PA-complete then it can in particular compute
a total extension of the universal $\{0,1\}$-valued partial-recursive
function, hence compute a list of total functions in which every
$\{0,1\}$-valued recursive function appears.

If $A$ is of hyperimmune degree there is an $A$-recursive function
that is not dominated by any recursive function.
This function can be used to compute infinitely many points
from every recursive set, in a uniform way.
More precisely, let $f\leq_T A$ be a function that is not
recursively dominated. If $\varphi_e$ is total then also
$\Phi(x) = \mu s. \varphi_{e,s}(x)\darrow$ is total,
hence $f(x)\geq \Phi(x)$  and $\varphi_{e,f(x)}(x)\darrow$ infinitely often.
For these $x$, let $P_e(x) = \varphi_e(x)$; for the other
$x$, let $P_e(x)=2$.
Then $P\leq_T A$ is a parametrisation and if $\varphi_e$ is total
then $P_e(x) = \varphi_e(x)$ infinitely often.

To see that
$\COMP$ $A$-subuniform does not imply that $A$ is PA-complete,
first note that PA-complete sets cannot have incomplete r.e.\
degree by a result of Scott and Tennenbaum \cite[p513]{Odifreddi}.
Second, by Theorem~\ref{Jockusch},
if $A$ is high then $\COMP$ is $A$-uniform.
So the non-implication follows from the existence of a
high incomplete r.e.\ set (Sacks, see \cite[p650]{OdifreddiII}).

To see that $A$ having hyperimmune degree
does not imply that $\COMP$ is $A$-subuniform,
note that again by Theorem~\ref{Jockusch}
we have for $A$ r.e.\ that  $\COMP$ is $A$-subuniform implies that $A$ is high.
Now let $A$ be r.e.\ nonrecursive (so that in particular
$A$ has hyperimmune degree \cite[p495]{Odifreddi}) and non-high.
Then $\COMP$ is not $A$-subuniform.
In particular, we see from this non-implication that
i.o.\ subuniformity of $\COMP$ does not imply subuniformity.

Finally, it is well-known that $A$ PA-complete does not
imply that $A$ has hyperimmune degree (and hence the weaker notions in the
diagram also do not imply it):
The PA-complete
sets form a $\Pi^0_1$ class, hence, since  the sets of hyperimmune-free
degree form a basis for $\Pi^0_1$ classes \cite[p509]{Odifreddi},
there is a PA-complete set of hyperimmune-free degree.
\end{proof}

\begin{pro} \label{iosubmeasure0}
Every i.o.\ subuniform class has recursive measure zero.
This relativises to:
If $\A$ is i.o.\ $X$-subuniform then $X$ covers $\A$.
\end{pro}

\begin{proof}
The ability to compute infinitely many bits from a set
clearly suffices to define a martingale succeeding on it.
The uniformity is just what is needed to make the
usual sum argument work.
\end{proof}

\begin{pro} \label{counterexample}
There exists a class of recursive sets that has recursive measure
zero and that is not i.o.\ subuniform.
\end{pro}

\begin{proof}
The class of all recursive sets $A$ satisfying
$\forall x\,[A(2x)=A(2x+1)]$ has recursive measure $0$
but is not i.o.\ subuniform: If $P$ would witness this class to be
i.o.\ subuniform then $Q$ defined as $Q_i(x) = \min\{P_i(2x),P_i(2x+1)\}$
would witness $\COMP$ to be i.o.\ subuniform, a contradiction.
\end{proof}

\medskip
\noindent
Above the recursive sets, the 1-generic sets are a natural
example of such a class that has measure zero but that is not
i.o.\ subuniform:  It is easy to see that the
1-generic sets have recursive measure zero because
for every such set $A$ there are infinitely many $n$ such
that $A\cap \{n,n+1,\ldots,2n\} = \emptyset$. On the other
hand, a variation of the construction in the proof of
Proposition~\ref{counterexample} shows that
the 1-generic sets are not i.o.\ $X$-subuniform for any~$X$:

\begin{pro} \label{counterexample2}
The 1-generic sets are not i.o.\ $X$-subuniform for any set~$X$.
\end{pro}

\begin{proof}
Let $P\subseteq\{0,1,2\}^\omega$ be an $X$-recursive parametrisation
and let $A$ be 1-generic relative to $X$
(so that $A$ is in particular 1-generic).
Then for every $n$, if $P_n(x)\neq 2$ for infinitely many $x$
then
\[
\big\{\,\sigma\in 2^{<\omega}: \exists x \,[P_n(x)\neq 2 \wedge P_n(x)\neq \sigma(x)]\,\big\}
\]
is $X$-recursive and dense,
hence $A$ meets this set of conditions and
consequently $P$ does not i.o.\ parametrise~$A$.
\end{proof}

\medskip
\noindent
Now both the example from Proposition~\ref{counterexample}
and the 1-generic sets are counterexamples to the
implication ``measure 0 $\Rightarrow$ i.o.\ subuniform''
because of the \emph{set structure} of the elements in the
class. One might think that for classes closed downwards under
Turing reducibility (that is, for classes defined by information
content rather than set structure) the situation could
be different. For example, one might conjecture that
for $\A$ closed downwards under Turing
reducibility, the implication
``$X$ covers $\A$ $\Rightarrow$ $\A$ i.o.\ $X$-subuniform''
would hold. Note that for $X$ recursive this is
not interesting, since any nonempty class closed downwards under
Turing reducibility contains $\COMP$ and $\COMP$ does not have
recursive measure zero. However, this conjecture is also not true:
Consider the class
\[
\A = \{A: A\leq_T G \mbox{ for some 1-generic } G\}.
\]
Clearly $\A$ is closed downwards under Turing reducibility and it follows
from
proofs by Kurtz~\cite{Kurtz} and by Demuth and Ku\v{c}era~\cite{DemuthKucera}
(a proof is also given by Terwijn~\cite{Terwijn1999}),
that $\A$ is a Martin-L\"of null set and that in particular the halting
problem $K$ covers $\A$.
However, by Proposition~\ref{counterexample2} the 1-generic sets are not
i.o.\ $K$-subuniform so that in particular $\A$ is not
i.o.\ $K$-subuniform.

\subsection{A Nonrecursive Set that does not Cover $\COMP$.}

\noindent
It follows from
Proposition~\ref{implications} and Proposition~\ref{iosubmeasure0}
that if $A$ is of hyperimmune degree then $A$ covers $\COMP$.
In particular every nonrecursive 
set comparable with $K$ covers $\COMP$.
We see that if $A$ cannot cover $\COMP$ then $A$ must have hyperimmune-free
degree. We now show that there are indeed nonrecursive sets 
that do not cover $\REC$. 
Indeed, the following result establishes that there are natural examples
of such sets; this result can be seen as a generalisation of
the result of Calude and Nies \cite{CaludeNies}
that Chaitin's $\Omega$ is wtt-complete and tt-incomplete; 
see Nies' book \cite[Theorem 4.3.9]{Ni09} for more information.

\begin{thm} \label{FS}
If $A$ is Martin-L\"of random then there is no martingale
$M \leq_{tt} A$ which covers \REC. In particular if $A$ is
Martin-L\"of random and of hyper\-immune-free Turing degree
then it does not cover \REC.
\end{thm}

\begin{proof}
Let $A$ be Martin-L\"of random and $M^A$ be truth-table reducible
to $A$ by a truth-table reduction which produces on every oracle
a savings martingale, that is, a martingale which never goes down
by more than $1$. Without loss of generality, the martingale starts
on the empty string with $1$, takes rational values
and is never less than or equal to~$0$.
Note that because of the truth-table property, one
can easily define the martingale $N$ given by
$$
   N(\sigma) = \int_{E \subseteq \omega} M^E(\sigma) \, dE,
$$
where the integration ``dE'' weights all oracles with the
uniform Lebesgue measure.
As one can replace the $E$ by the strings up to $\use(|\sigma|)$
using the recursive use-function $\use$ of the truth-table reduction,
one has that
$$
   N(\sigma) = \sum_{\tau \in \{0,1\}^{\use(|\sigma|)}}
               2^{-|\tau|} M^\tau(\sigma)
$$
and $N$ is clearly a recursive martingale; also the values
of $N$ are rational numbers.
Let $B$ be a recursive set which is adversary to $N$,
that is, $B$ is defined inductively such that
$$
\forall n\,[N(B\restr(n+1)) \leq N(B\restr n)].
$$
Define the uniformly r.e.\ classes $S_n$ by
$$
   S_n = \{E: M^E \mbox{ reaches on $B$ a value beyond } 2^n+1\}.
$$
By the savings property, once $M^E$ has gone beyond $2^n+1$ on $B$,
$M^E$ will stay above $2^n$ afterwards.
It follows that the measure of these $E$ can be at most $2^{-n}$. 
So $\mu(S_n) \leq 2^{-n}$
for all $n$ and therefore the $S_n$ form a Martin-L\"of test.
Since $A$ is Martin-L\"of random, there exists $n$ such that $A\notin S_n$, 
and hence $M^A$ does not succeed on $B$. 
\end{proof}

\medskip
\noindent
The anonymous referee pointed out to the authors that Theorem~\ref{FS}
has a variant which is true when $A$ is Kurtz random. The precise statement
is the following:

\begin{prop}
\begin{enumerate}[{\bf(a)}]
\item If $A$ is Kurtz random then there is no martingale $M \leq_{tt} A$ and no
recursive bound function $f$ which Kurtz cover $\COMP$,
that is, which satisfy that for all $B \in \COMP$ and almost all $n$,
$M(B(0)B(1)\ldots B(f(n))) \geq n$.
\item If $A$ is Schnorr random then there is no martingale $M \leq_{tt} A$ and
no recursive bound function $f$ which Schnorr cover $\COMP$,
that is, which satisfy that for all $B \in \COMP$ and for infinitely many $n$,
$M(B(0)B(1)\ldots B(f(n))) \geq n$.
\end{enumerate}
\end{prop}

\noindent
Note that a weakly $1$-generic set $A$ is Kurtz random
and coincides on arbitrarily long parts with any given recursive
set $B$, so the martingale which bets half of the capital on the next digit
of $A$ and $B$ to be the same will Schnorr cover all recursive sets $B$.
Therefore one has to use ``Kurtz cover'' for part (a). The observation
of the referee allows then to conclude that the truth-table degrees of
weakly $1$-generic and Schnorr random sets can never be the same, as the
first ones Schnorr cover $\COMP$ and the second ones don't.

We note that the set
$\bigset{A \in \{0,1\}^\omega: A \mbox{ covers } \COMP}$ has measure~1.
This follows from Proposition~\ref{implications} and the fact that the hyperimmune
sets have measure~1 (a well-known result of Martin, 
see \cite[Theorem 8.21.1]{DH}). 

We note that apart from the hyperimmune degrees, there are other degrees
that cover~\COMP.

\begin{pro}
There are sets of hyperimmune-free degree that cover the class $\COMP$.
\end{pro}

\begin{proof}
As in Proposition~\ref{implications}, take a PA-complete set $A$ of
hyperimmune-free degree. Then the recursive sets are
$A$-subuniform, so by Proposition~\ref{iosubmeasure0} $A$ covers $\COMP$.
\end{proof}

\subsection{Computing Covers versus Uniform Computation}

\noindent
We have seen above that in general the implication
``$X$ covers $\A$ $\Rightarrow$ $\A$ i.o.\ $X$-subuniform''
does not hold, even if $\A$ is closed downwards under Turing reducibility.
A particular case of interest is whether there are sets that can cover 
$\COMP$ but relative to which $\COMP$ is not i.o.\ subuniform.

\begin{thm} \label{separation}
There exists a set $A$ that Schnorr covers $\COMP$ but relative to which
$\COMP$ is not i.o.\ $A$-subuniform.
\end{thm}

\begin{proof}
We construct the set $A$ by choosing a total extension of hyperimmune-free
Turing degree of a partial-recursive $\{0,1\}$-valued function $\psi$
built by a finite injury construction. 
In the following, we will consider parametrisations computable by $A$. 
Because $A$ is of hyperimmune-free degree, for every Turing reduction to $A$ 
there is an equivalent truth-table reduction to $A$
by Proposition~\ref{prop:tt}, 
so it will be sufficient to only consider the latter. 
We will consider tt-reductions $\Phi^E$ that compute 
i.o.\ parametrisations relative to an oracle $E$, and we denote 
the $i$-th component of such a parametrisation by~$\Phi_i^E$.

Let $I_0 = \{0,1,2\}$ and, for $n>0$, $I_n = \{3^n,3^n+1,\ldots,3^{n+1}-1\}$.
Now if $A$ coincides with a set $B$ on infinitely many intervals $I_n$ then
the martingale which always puts half of its money onto the next bit
according to the value of $A$ succeeds on $B$, indeed, it even
Schnorr succeeds on $B$. The reason is that if $I_n$ is such an
interval of coincidence, then at least $3^{n+1}-3^n$ of the bets
are correct and the capital is at least 
$3^{3^{n+1}-3^n}/2^{3^{n+1}} = (9/8)^{3^n}$, 
as it multiplies with $3/2$ at a correct bet and halves
at an incorrect bet. Thus the overall goal of the construction
is to build a partial recursive function $\psi$ with the following 
properties:
\begin{itemize}
\item $\psi$ coincides with every recursive set on infinitely many $I_n$,
      and therefore every total extension $E$ of $\psi$ Schnorr covers $\COMP$;
\item For every truth-table reduction $\Phi$ there is a recursive set $B$
      such that $\{B\}$ is not i.o.\ subuniform for any total extension $E$
      of $\psi$ via $\Phi$:
      $$ (*) \ \ \ \ 
         \forall i\,(\forall {}^\infty x\,[\Phi_i^E(x) = 2] \vee
                     \exists x\,[\Phi_i^E(x) = 1-B(x)]).
      $$
\end{itemize}
To simplify the construction, we define a list of admissible 
truth-table reductions, which are all truth-table reductions $\Phi$ 
that satisfy one of the following two conditions:
\begin{enumerate}[(1)]
\item \label{either}$\Phi$ is total for all oracles and computes a sequence
      $\Phi_0,\Phi_1,\ldots$ of $\{0,1,2\}$-valued functions such that
      for all oracles $E$, $\Phi_{2i}^E$ is the characteristic function
      of the $i$-th finite set.
\item\label{or} $\Phi^E_i$ is partial for all $i$ and all $E$, and the set
      $\{(i,x): \Phi^E_i(x)$ is defined for some $E\}$ is finite.
\end{enumerate}
It is easy to see that there is an effective list of all admissible
truth-table reductions. Condition (\ref{either}) includes
that all $\Phi_i^E$ for even $i$ follow finite sets; this is needed
in order to avoid that the construction of the set $B$ gets stuck;
it is easy to obtain this condition by considering a join of a given
truth-table reduction with a default one computing all finite sets.
The inclusion of the partial reductions is there in order to account for
the fact that there is no recursive enumeration of all total recursive
functions and thus also no recursive enumeration of all total truth-table
reductions. So Condition (\ref{or}) is needed to make the enumeration effective.

There will be actions with different priority; whenever several actions
apply, the one with the highest priority represented by the lowest
natural number will be taken. Here $\psi_s$ at stage $s$ is for each
interval either defined on the whole interval or undefined on the whole
interval, and it is defined only on finitely many intervals; furthermore,
$J_n$ refers to the $n$-th interval where $\psi_s$ is undefined and
$c_n$ refers to the number of arguments where $\psi_s$ is undefined
below $\min(J_n)$. The following actions can be taken at stage $s$
with priority $n = \max\{i,j\}$, for the parameters $i,j$ given below:
\begin{enumerate}
\item The action {\em requires attention\/} if $\varphi_{i,s}(x)$ is defined
      for all $x \leq \max(J_n)$ and
      if there are exactly $j$ intervals $I_k$ with $\max(I_k) < \min(J_n)$
      and $\psi_s(y)\!\downarrow = \varphi_{i,s}(y)\!\downarrow$
      for all $y \in I_k$. In this case the action requires attention with
      priority~$n$. If the action receives attention then we let 
      $\psi_{s+1}(x) = \varphi_{i,s}(x)$ for all $x \in J_n$;
      thus each $J_m$ with $m \geq n$ will move to $J_{m+1}$.
\item Let $\Phi$ be the $j$-th admissible truth-table reduction and
      let $x$ be the least value where the set $B$ defined alongside $\Phi$
      (to satisfy $(*)$ above) has not yet been defined by stage $s$, 
      and let $E$ be a total extension of $\psi_s$.
      The action requires attention if the following three conditions hold:
      \begin{itemize}
      \item $\Phi_k^F(y)$ is defined for all oracles $F$,
            all $y \leq x$ and all $k \leq n$ by stage $s$ with use $s$,
      \item $\forall y < x \,[\Phi_i^E(y) \in \{B(y),2\}]$,
      \item $\Phi_i^E(x) \in \{0,1\}$.
      \end{itemize}
      If it receives attention then one defines
      $B(x) = 1 - \Phi_i^E(x)$ and for all $m \geq n$ with
      $\min(J_m) \leq s$ and all $y \in J_m$ one defines
      $\psi_{s+1}(y) = E(y)$ and therefore all the intervals
      $J_m$ with $m \geq n$ are moved beyond $s$.
\end{enumerate}
At stage $s$ the algorithm chooses an action with highest priority (= least
numerical value of the priority number) that can be taken (if any), and the 
algorithm does not change anything if there is no action which can be taken;
in the case that several actions can be taken with the same highest
priority, it uses some default ordering (length-lexicographical ordering
of some coding of the actions) in order to decide which one to do.
As mentioned at the beginning, we let $A$ be any extension of the so 
constructed function $\psi$ of hyperimmune-free Turing degree. 
Such $A$ exists by the standard construction of a hyperimmune-free degree by 
Miller and Martin, see \cite{Odifreddi}.

The first part of the verification consists of inductively proving
the following for each~$n$:
\begin{itemize}
\item The number 
      $c_n = \sum_{m<n} |J_m|$ increases only finitely often, 
      and after some stage $s$, no action of priority $m<n$
      is taken and none of the intervals $J_m$ with $m<n$ moves again;
\item After this stage $s$, the number of times that an action of
      priority $n$ will be taken is at most $(2^{c_n}+1) \cdot (2n+1)$,
      and the interval $J_n$ will only be moved when an action of
      priority $n$ acts, that is, it will also be moved only finitely often.
\end{itemize}
Note that the first item is the induction hypothesis and the proof
of the second is the inductive step; for $n=0$ the first hypothesis
is void and therefore satisfied. When an action of priority $n$ is taken,
only the values of $\psi$ in some intervals $J_m$ with $m \geq n$ will
be filled, and therefore only the intervals $J_n$ and beyond will be moved.

First consider actions of priority $n$ which are of type $1$, that is,
for which $J_n$ gets defined according to some $\varphi_i$. Here
$n = \max\{i,j\}$ where $j$ is the number of intervals $I_k$ below $J_n$
where $\psi$ and $\varphi_i$ are both defined by stage $i$ and equal.
Whenever the action is taken and $J_n$ is moved afterwards,
the number $j$ increases by $1$; hence it happens at most $2n+1$ times
that $\psi$ is defined on $J_n$ to be equal to some $\varphi_i$ 
by an action of type~$1$: For each $i$ there is one action in the
case that $i<n$ and $j = n$, and $n+1$ actions when $i=n$ and
$j=0,1,\ldots,n$.

Now consider actions of priority $n$ which are of type $2$.
For each pair $(i,j)$ with $\max\{i,j\}=n$ and each $E$ extending $\psi$,
there is at most
one action where the $j$-th admissible truth-table reduction $\Phi$ applies
and $B(x)$ is set to be different from the value $\Phi_i^E(x)$ because
all $y<x$ satisfy $\Phi_i^E(y) \in \{B(y),2\}$; furthermore, if this
action is made subsequently for two different sets $E$ and $F$
with the same $(i,j)$ then $E(z) \neq F(z)$ for some $z \in \cup_{m<n} J_m$
(this set does not change after stage $s$) and therefore there are for
each $(i,j)$ at most $2^{c_n}$ such actions, giving the overall
upper bound of $2^{c_n} \cdot (2n+1)$ actions of type $2$ with priority $n$
after stage $s$.

The sum of the two calculated upper bounds
gives the overall upper bound $(2^{c_n}+1) \cdot (2n+1)$
of actions of priority $n$ carried out after stage $s$.
Thus the inductive step is completed.

It is clear that actions of the first type are carried out eventually
for all $n$ in the case that $\varphi_i$ is total and $\{0,1\}$-valued.
Thus each recursive set coincides with each total extension of $\psi$
on infinitely many intervals; in particular $A$ does so and therefore
$A$ is weakly Schnorr covering, as explained at the beginning of the proof.

Now consider the $j$-th admissible truth-table reduction and assume
that it is total. We observe that the set $B$ gets defined for
every $x$, as for each $x$ there exists an $i$ such that $\Phi_i$ is
$\{0,1\}$-valued and coincides with $B$ below $x$, hence $B(x)$ will
be defined and either diagonalise $\Phi_i^A$ at $x$ or diagonalise
some other set with some other oracle. Choose any $i$ and let $s$
be so large that no action of priority $\max\{i,j\}$ or less will
take place at or after stage $s$. Then there is no $x > s$ such
that the following two conditions are satisfied at the state $t$
where $B(x)$ gets defined:
\begin{itemize}
\item $\Phi^A_i(x) \in \{0,1\}$ and
\item $\forall y < x\,[\Phi^A_i(y) \in \{B(y),2\}]$.
\end{itemize}
The reason is that if these two conditions would be satisfied then
an action of priority $\max\{i,j\}$ would qualify and enforce that some
action of this or higher priority has to be carried out; by assumption
on $s$ this does however not happen. Therefore either $\Phi^A_i$
is inconsistent with $B$ and there is an $x$ with $\Phi_i^A(x) = 1-B(x)$,
or all $x>s$ satisfy $\Phi^A_i(x) = 2$. 
We conclude that the $j$-th admissible reduction $\Phi$ does not 
witness that $\COMP$ is i.o.\ $A$-subuniform and, as $j$ was arbitrary,
$\COMP$ is not i.o.\ $A$-subuniform.
\end{proof}

\medskip
\noindent
Theorem~\ref{separation} shows in particular that there are 
sets $A$ covering $\REC$ for which $\REC$ is not i.o.\ $A$-subuniform. 
Thus we see that there are sets that cover $\COMP$ 
``truly probabilistically''. 

Theorem~\ref{separation} also has a counterpart.
Recall that Terwijn and Zambella \cite{TZ01} showed 
(reformulating using the results of Franklin~\cite{Fr08})
that no Schnorr trivial $A$ of hyperimmune-free Turing degree is
weakly Schnorr covering; actually they showed that every class of sets Schnorr
covered by an $A$-recursive martingale $M$ with bound $f$
is already Schnorr covered by a recursive $N$ with recursive
bound $g$, and the just mentioned observation follows from
the fact that there is no recursive martingale covering all recursive sets.

The next result completes the picture from Theorem~\ref{separation} 
that the two notions ``$\COMP$ is i.o.\ $A$-subuniform'' and 
``$A$ is weakly Schnorr covering'' (which means that $A$ Schnorr covers $\REC$)
are incomparable. Note that they are both implied by PA-complete and by
hyper\-immune and they both imply ``$A$ covers $\REC$''.

\begin{thm} \label{separation2}
There exists a Schnorr-trivial and hyper\-immune-free set $A$
such that $\COMP$ is i.o.\ $A$-subuniform but $A$
does not Schnorr cover $\COMP$.
\end{thm}

\begin{proof}
Let $c_\Omega$ be the modulus of convergence
of Chaitin's $\Omega$.
The function $c_\Omega$ is approximable from below and
dominates every recursive function.
We construct a recursive function $f$ and a partial recursive
$\{0,1\}$-valued function $\psi$ such that the following
conditions are met:
\begin{itemize}
\item For each $e$ there is at most one $n$ with
      $f(n) = e$ and $\psi(f(n))$ being undefined;
\item For each total $\{0,1\}$-valued
      $\varphi_e$ there are infinitely many $n$
      with $f(n) = e$;
\item For all $n$, if $c_\Omega(f(n)) \geq n$ then $\psi(n)$
      is defined else $\psi(n)\!\simeq\!\Parity(\varphi_{f(n)}(n))$.
\end{itemize}
The oracle $A$ will then be
fixed as a hyperimmune-free total extension of $\psi$.

The recursive function
$f$ can be defined inductively with monitoring $\psi$
on the places below $n$ where $f(n)$ is
the coordinate $e$ of the least pair $(d,e)$ such that
there are exactly $d$ many $m < n$ with $f(m) = e$
and $\psi(m)$ being defined for each of these $m$ within
$n$ computation steps. There are the following two cases:
\begin{itemize}
\item There are infinitely many $n$ with $f(n) = e$.
      Then $\psi(n)$ is defined for all these~$n$ and
      $\psi(n) = \Parity(\varphi_e(n))$ for almost all of these $n$.
\item There are only finitely many $n$ with $f(n) = e$.
      Then $\psi(n)$ is undefined only on the largest of these $n$
      and this $n$ also satisfies that $\varphi_e(n)$ is undefined.
\end{itemize}
Furthermore, we define $\psi(n)$ by the first of the
following two searches that halt:
\begin{itemize}
\item If $\varphi_{f(n)}(n)$ converges then one tries to define that
      $\psi(n)$ is $\Parity(\varphi_{f(n)}(n))$;
\item If $c_\Omega(f(n)) \geq n$ then one tries to define that $\psi(n)=0$.
\end{itemize}
Thus if $\varphi_e$ is $\{0,1\}$-valued and total then the
first case applies and $\varphi_e(n) = \psi(n)$ for almost
all $n$ where $f(n) = e$.

One now makes a family $P_d$ consisting of all finite variants of
functions $Q_e$ which defined which are defined as follows:
If $f(n) = e$ then $Q_e(n) = A(n)$
else $Q_e(n) = 2$. Note that the $Q_e$ are uniformly recursive
in $A$ and so are the $P_d$. Furthermore, as for each total
and $\{0,1\}$-valued $\varphi_e$
the function $Q_e$ is correct on almost all of its infinitely many
predictions, one finite variant $P_d$ of $Q_e$ will coincide
with $\varphi_e$ on all of its predictions. Thus $\COMP$ is
i.o.\ $A$-subuniform.

The function $\psi$ has below $c_\Omega(e)$ only
undefined places at $n$ with $f(n) < e$ and for
each possible value of $f$ below $e$ at most
one undefined place, hence the domain of $\psi$ is
dense simple (see \cite{OdifreddiII} for the definition). 
By a result of Franklin and Stephan \cite{FS10}, the
total extension $A$ of $\psi$ is Schnorr trivial.
As $A$ has also hyperimmune-free degree, $A$ is not
weakly Schnorr covering, by the results of 
Terwijn and Zambella~\cite{TZ01} discussed above. 
\end{proof}

\medskip
\noindent
This construction has a relation to \cite[Question 4.1(8)]{Brendle.Brooke.ea:14}
which could be stated as follows:
\begin{quote}
Is there a $\DNR$ and hyperimmune-free set which neither computes a Schnorr
random nor is weakly Schnorr covering?
\end{quote}
Note that the original question
of the authors asked for weakly meager covering in place of $\DNR$; however,
weakly meager covering together with not weakly Schnorr covering implies
both $\DNR$ and hyperimmune-free while, for the other way round, $\DNR$ implies
weakly meager covering. Thus the formulation given here is equivalent
to the original question.

So let $f$, $\psi$ be as in the proof of Theorem \ref{separation2}, and
let $\vartheta$ be the following numbering:
If $n \neq m$ then $\vartheta_n(m) = \varphi_{f(n)}(m)$ else
$\vartheta_{n}(n)$ is obtained by monitoring the definition of
$\psi$ and doing the following:
\begin{itemize}
\item If $\psi(n)$ gets defined by following $\Parity(\varphi_{f(n)}(n))$
      then $\vartheta_n(n) = \varphi_{f(n)}(n)$;
\item If $\psi(n)$ gets defined by taking $0$ due to
      $c_\Omega(f(n)) \geq n$ then $\vartheta_n(n) = 0$;
\item If $\psi(n)$ does not get defined then $\vartheta_n(n)$
      remains undefined.
\end{itemize}
Now consider the $K$-recursive function $h$ given as follows:
$h(e)$ is the first $n$ such that $f(n) = e$ and
$\vartheta_{f(n)}(n) \simeq \varphi_e(n)$ --- the halting problem
$K$ allows us to check this. The construction gives that such an index
is always found and therefore $\vartheta_{h(e)} = \varphi_e$.
A numbering with such a $K$-recursive translation function is
called a $K$-acceptable numbering.
Furthermore, the mapping $n \mapsto 1-A(n)$ witnesses that
$A$ is $\DNR^\vartheta$. Thus $A$ satisfies the conditions
from \cite[Question 4.1(8)]{Brendle.Brooke.ea:14}
with $\DNR^{\vartheta}$ in place
of $\DNR$; this does not answer the original question, as
$\DNR^\vartheta$ is weaker than $\DNR$.

\begin{cor}
For some $K$-acceptable numbering $\vartheta$,
there is a $\DNR^\vartheta$, Schnorr trivial and hyperimmune-free oracle $A$;
such an $A$ neither computes a Schnorr random nor is weakly Schnorr covering.
\end{cor}

\noindent
We can now extend the picture of Proposition~\ref{implications}
to the following.

\begin{thm}
We have the following picture of implications:
$$
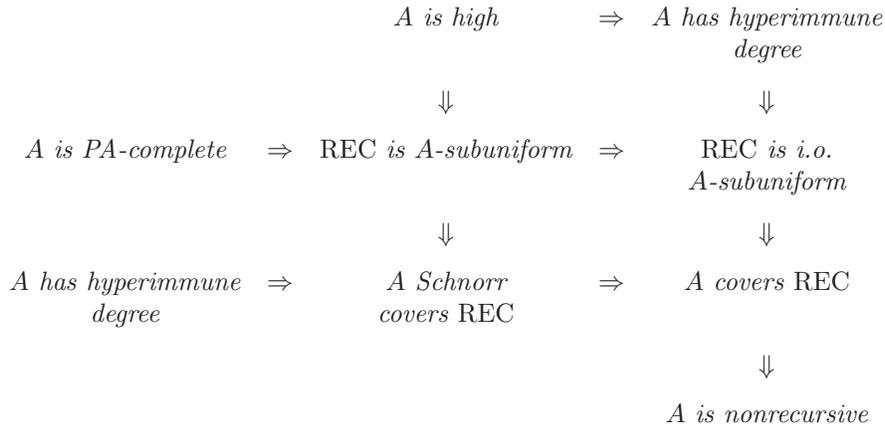
\begin{array}{ccccc}
&& \mbox{$A$ is high} & \Rightarrow & \mbox{$A$ has hyperimmune} \\
&&&& \mbox{degree} \\
&& \raisebox{0pt}[16pt][10pt]{$\Downarrow$}  
&& \raisebox{0pt}[16pt][10pt]{$\Downarrow$}   \\
\mbox{$A$ is PA-complete} & \Rightarrow &
\COMP \mbox{ is $A$-subuniform} & \Rightarrow & \mbox{$\COMP$ is i.o.} \\
&&&& \mbox{$A$-subuniform} \\
&& \raisebox{0pt}[16pt][10pt]{$\Downarrow$}  
&& \raisebox{0pt}[16pt][10pt]{$\Downarrow$}   \\
\mbox{$A$ has hyperimmune} & \Rightarrow & \mbox{$A$ Schnorr} &
  \Rightarrow & \mbox{$A$ covers $\COMP$} \\
\mbox{degree} && \mbox{covers $\COMP$} && \\
&&&& \raisebox{0pt}[16pt][10pt]{$\Downarrow$}   \\
&&&& \mbox{$A$ is nonrecursive} 
\end{array}
$$
\noindent
No other implications hold besides the ones indicated; note that for
having a clean graphical presentation, the notion
``$A$ has hyperimmune degree'' has two entries.
\end{thm}

\begin{proof}
The upper part of the diagram was discussed in Proposition~\ref{implications}.
That $\COMP$ i.o.\ $A$-subuniform implies that $A$ covers $\COMP$ is
immediate from Proposition~\ref{iosubmeasure0}.
That $\REC$ i.o.\ $A$-subuniform does not imply that $A$ Schnorr covers $\REC$
was proven in Theorem~\ref{separation}.
That the converse also does not hold was proven in Theorem~\ref{separation2}.

Since by Proposition~\ref{implications}, $A$ is PA-complete does not
imply that $A$ has hyperimmune degree; the same is true for all notions
implied by $A$ being PA-complete, that is, for $\COMP$ being $A$-subuniform,
for $A$ Schnorr covering $\COMP$, for $\COMP$ being i.o.\ $A$-subuniform,
for $A$ covering $\COMP$ and for $A$ being nonrecursive. Rupprecht
\cite{Rupprecht,Rupprechtthesis} proved that sets
of weakly $1$-generic degree -- which
are the same as sets of hyperimmune degree -- are Schnorr covering $\COMP$.

That $A$ nonrecursive does not imply that $A$ covers $\COMP$
follows from Theorem~\ref{FS}. 
\end{proof}

\medskip
\noindent
The following interesting question is still open.

\begin{que} \label{question2}
Are there sets $A$ such that $A$ covers $\COMP$, but not the class $\CE$
of all recursively enumerable sets?
\end{que}

\noindent
Note that this really asks for the class of all recursively enumerable
sets and not the class of all left-r.e.\ sets; if $A$ is low for
Martin-L\"of randomness and nonrecursive then $A$ covers $\COMP$
but fails to cover the left-r.e.\ sets, as $\Omega$ is Martin-L\"of
random relative to $A$.

Hirschfeldt and Terwijn~\cite{HirschfeldtTerwijn} proved that the low
sets do not have $\Delta^0_2$-measure zero in $\Delta^0_2$,
that is, there does not exist a $K$-recursive martingale
that succeeds on all the low sets, where $K$ is the halting-problem.

The reason is that given such an martingale $M^K$, one can consider
the variant $O^K$ which behaves on the bits with index $2k$ like $M^K$ on
the bits with index $k$ and which ignores the bits with indices $2k+1$;
furthermore, let $N^K$ be a martingale which covers on all sets
which are not Martin-L\"of random. The sum of $N^K$ and $O^K$ gives
a $K$-recursive martingale which covers all sets covered by
$O^K$ or $N^K$. However, some $K$-recursive set $A \oplus B$
withstands this sum martingale. Thus $A \oplus B$ is Martin-L\"of random.
By van Lambalgen's Theorem,
the half $A$ of $A \oplus B$ is low and Martin-L\"of random;
as $A$ consists in $A \oplus B$ of the bits with index $2k$,
the construction gives that $M^K$ does not cover $A$.

In particular, the low Martin-L\"of random sets are not i.o.\ $K$-subuniform.
Despite the non-uniformity of the
low sets, Downey, Hirschfeldt, Lempp and Solomon \cite{DHLS} succeeded in
constructing a set in $\Delta^0_2$ that is bi-immune for the low sets.

\bigskip\noindent
{\bf Acknowledgements.}
The authors would like to thank George Barmpalias and Michiel van Lambalgen 
for discussions about Section~\ref{sec:iosub} and
Andr\'e Nies and Benoit Monin for correspondence and thorough
checking of Theorem~\ref{firstStatement}.

\end{document}